\documentclass[12pt,twosides]{amsart}
\usepackage{amssymb,amsmath,amsthm, amscd, enumerate, mathrsfs}
\usepackage{graphicx, hhline}
\usepackage[all]{xy}
\usepackage[usenames]{color}
\usepackage{hyperref}
\usepackage{fancyhdr}
\usepackage[top=30truemm,bottom=30truemm,left=30truemm,right=30truemm]{geometry}

\hypersetup{colorlinks=true}

\title{Log Iitaka conjecture for abundant log canonical fibrations}
\author{Kenta Hashizume}
\date{2019/04/29, version 0.07}
\address{Graduate School of Mathematical Sciences, 
The University of Tokyo, 3-8-1 Komaba Meguro-ku Tokyo 153-8914, Japan}
\email{hkenta@ms.u-tokyo.ac.jp}

\newtheorem{thm}{Theorem}[section]

\newtheorem{lem}[thm]{Lemma}
\newtheorem{cor}[thm]{Corollary}

\newtheorem{conj}[thm]{Conjecture}

\theoremstyle{definition}
\newtheorem{defn}[thm]{Definition}
\newtheorem{rem}[thm]{Remark}

\newtheorem*{ack}{Acknowledgments}

\newtheorem{step}{Step}

\newtheorem*{claim*}{Claim}
\begin{document}

\maketitle

\begin{abstract}
We prove that the log Iitaka conjecture holds for log canonical fibrations when log canonical divisor of a sufficiently general fiber is abundant.  
\end{abstract}

\tableofcontents 

\section{Introduction}
Throughout this article we will work over the complex number field. 

Canonical divisor of smooth projective varieties is an important object to study geometric properties of the varieties. 
In the birational geometry, we expect that we can classify all smooth projective varieties by using birational invariants defined with the canonical divisor.
Kodaira dimension is one of the birational invariants, and the Iitaka conjecture is a fundamental problem about Kodaira dimension. 
For any fibration of smooth projective varieties, the Iitaka conjecture asserts that the sum of the Kodaira dimension of base variety and that of the general fiber is not greater than the Kodaira dimension of total variety.
Kawamata \cite{kawamataiitakacurve} proved the conjecture when the base variety is a curve.
He also  proved the conjecture when the general fiber has a good minimal model  (\cite{kawamataiitakafiber}). 
Viehweg \cite{viehweg} proved the conjecture when the base variety is of general type, and Birkar \cite{birkariitaka} proved the conjecture when the dimension of total variety is less than or equal to $6$.  

In a viewpoint of the minimal model theory and a viewpoint of the Iitaka conjecture for morphisms between open varieties, it is natural to consider the logarithmic analogue of the Iitaka conjecture.

\begin{conj}[{\cite[Conjecture 1.2]{fujinoiitaka}}]\label{conjiitaka}
Let $f\colon X\to Y$ be a surjective morphism of smooth projective varieties with connected fibers. 
Let $\Delta_{X}$ (resp.~$\Delta_{Y}$) be a reduced simple normal crossing divisor on $X$ (resp.~$Y$) such that ${\rm Supp}f^{*}\Delta_{Y}\subset {\rm Supp}\Delta_{X}$. 
Let $F$ be a sufficiently general fiber of $f$ and we define $\Delta_{F}$ by equation $K_{F}+\Delta_{F}=(K_{X}+\Delta_{X})|_{F}$.
Then we have
$$\kappa(X,K_{X}+\Delta_{X})\geq \kappa(F,K_{F}+\Delta_{F})+\kappa(Y,K_{Y}+\Delta_{Y}).$$
\end{conj}

Conjecture \ref{conjiitaka} holds true if we assume the minimal model theory for all lc pairs (see \cite[Theorem 1.3]{fujinoiitaka} and \cite{fujino-subadd-correct}). 
Currently, Conjecture \ref{conjiitaka} is known in the following cases:
\begin{itemize}
\item
$K_{F}+\Delta_{F}$ is big (\cite{kp}), 
\item
$K_{Y}+\Delta_{Y}$ is big (\cite{fujino-noteiitaka}), 
\item
$Y$ has maximal albanese dimension, $\Delta_{Y}=0$ and $K_{F}+\Delta_{F}$ is abundant (\cite{huiitaka}). 
\end{itemize}

In this note, we prove the following theorem. 

\begin{thm}\label{thmlogiitaka}
Conjecture \ref{conjiitaka} holds true when $K_{F}+\Delta_{F}$ is abundant. 
In particular, Conjecture \ref{conjiitaka} holds true when $(F,\Delta_{F})$ has a good minimal model. 
\end{thm}

For abundant divisors, see Definition \ref{defnabundant}. 
Theorem \ref{thmlogiitaka} is a partial generalization of \cite[Corollary 1.2]{kawamataiitakafiber} to log pairs. 

As a corollary, we obtain the following statement.

\begin{cor}\label{coriitaka} 
Conjecture \ref{conjiitaka} holds true when 
${\rm dim}X-{\rm dim}Y\leq 3$ or
${\rm dim}X\leq5$. 
\end{cor} 

A key ingredient to prove Theorem \ref{thmlogiitaka} is the following main result of this note, which is a generalization of \cite[Theorem 9.9]{kp}. 

\begin{thm}\label{thmmain}
Let $f\colon X\to Y$ be a surjective morphism from a normal projective variety to a smooth projective variety with connected fibers, and let $(X,\Delta)$ be an lc pair such that $\Delta$ is a $\mathbb{Q}$-divisor. 
We set $K_{F}+\Delta_{F}=(K_{X}+\Delta)|_{F}$, where $F$ is a sufficiently general fiber of $f$. 
Let $M$ be a $\mathbb{Q}$-Cartier divisor on $Y$. 
Suppose that 
\begin{itemize}
\item
$\kappa(Y,M)\geq0$, 
\item
$\kappa(F,K_{F}+\Delta_{F})\geq0$, and 
\item
$K_{F}+\Delta_{F}$ is abundant. 
\end{itemize}
Then, the followings hold true:
\begin{enumerate}
\item
The inequality $\kappa(X,K_{X}+\Delta-f^{*}K_{Y}+f^{*}M)\geq \kappa(F,K_{F}+\Delta_{F})+\kappa(Y,M)$ holds.  
\item 
Put $n={\rm dim}Y-\kappa(Y,M)$, and suppose in addition that all projective klt pairs of dimension $\leq n$ have good minimal models or Mori fiber spaces.  
If $M-K_{Y}$ is nef and abundant, then $K_{X}+\Delta-f^{*}K_{Y}+f^{*}M$ is abundant.
\end{enumerate} 
\end{thm}

To prove Theorem \ref{thmlogiitaka} we only need Theorem \ref{thmmain} (1). 
Theorem \ref{thmmain} (2) is a result of the generalized abundance for generalized lc pairs. 
For the proof of Theorem \ref{thmmain}, we use a result in \cite{kp} and apply argument as in \cite{huiitaka}. 
For details, see Section \ref{sec3}. 


The contents of this article are as follows: 
In Section \ref{sec2}, we collect definitions and recall known results on fibrations from an lc pair to a variety. 
In Section \ref{sec3}, we prove Theorem \ref{thmmain}, Theorem \ref{thmlogiitaka} and Corollary \ref{coriitaka}.

\begin{ack}
The author was partially supported by JSPS KAKENHI Grant Number JP16J05875. 
The author would like to express his gratitude to Dr. Sho Ejiri for stimulating discussion. 
The author is grateful to Professor Osamu Fujino for answering questions and giving comments. 
Especially, he kindly allows the author to use the argument in the proof of Corollary \ref{coriitaka}.  
The author thanks Dr. Takahiro Shibata for answering questions. 
\end{ack}

\section{Preliminaries}\label{sec2}

We will freely use the definitions of singularities of pairs as in \cite{kollar-mori}. 

Let $f\colon X\to Y$ be a projective morphism of normal varieties. 
We call $f$ a {\em contraction} if it is surjective and has connected fibers. 

Let $X$ be a normal projective variety and $D$ be a $\mathbb{Q}$-Cartier divisor on $X$. 
We denote the Iitaka dimension of $D$ by $\kappa(X,D)$. 

Next, we recall definition of the numerical dimension for $\mathbb{R}$-Cartier $\mathbb{R}$-divisors and abundant $\mathbb{Q}$-divisors. 

\begin{defn}[Numerical dimension]\label{defnnum}
Let $X$ be a normal projective variety and $D$ be an $\mathbb{R}$-Cartier $\mathbb{R}$-divisor on $X$. 
We define the {\em numerical dimension} of $D$, denoted by $\kappa_{\sigma}(X,D)$, as follows: 
For any Cartier divisor $A$ on $X$, we set
$$
\sigma(D;A)={\rm max}\left\{k\in \mathbb{Z}_{\geq0}\middle|\, \underset{m\to \infty}{\rm lim}{\rm sup}\frac{{\rm dim}H^{0}(X,\mathcal{O}_{X}(\llcorner mD \lrcorner+A))}{m^{k}}>0\right\}
$$
if ${\rm dim}H^{0}(X,\mathcal{O}_{X}(\llcorner mD \lrcorner+A))>0$ for infinitely many $m>0$ and otherwise we set $\sigma(D;A):=-\infty$. 
Then, we define 
$$\kappa_{\sigma}(X,D):={\rm max}\{\sigma(D;A)\,|\,A{\rm \; is\; Cartier}\}.$$
\end{defn}

Note that we have $\kappa(X,D)\leq \kappa_{\sigma}(X,D)$ for any $\mathbb{Q}$-Cartier divisor $D$. 

\begin{defn}[Abundant divisor]\label{defnabundant}
Let $X$ be a normal projective variety and $D$ be a $\mathbb{Q}$-Cartier divisor on $X$. 
We say $D$ is {\em abundant} when the equality $\kappa(X,D)=\kappa_{\sigma}(X,D)$ holds.
\end{defn}


The following lemma easily follows from definition. 

\begin{lem}
\label{lemdim}
Let $X$ be a normal projective variety and $D$ be a $\mathbb{Q}$-Cartier divisor on $X$. 
Suppose that there is an effective $\mathbb{Q}$-divisor $E$ such that $\kappa(X,D-E)\geq0$. 

Then $\kappa(X,D)=\kappa(X,D-tE)$ and $\kappa_{\sigma}(X,D)=\kappa_{\sigma}(X,D-tE)$ for any $t\in (0,1)$. 
\end{lem}


The following lemma is used in this paper. 

\begin{lem}\label{lem--iitakafib}
Let $(X,\Delta)$ be a projective lc pair with a boundary $\mathbb{Q}$-divisor $\Delta$. 
Suppose that $K_{X}+\Delta$ is abundant and $\kappa(X,K_{X}+\Delta)\geq 0$. 
Let $X\dashrightarrow V$ be the Iitaka fibration of associated to $K_{X}+\Delta$. 
Pick a log resolution $f\colon Y\to X$ of $(X,\Delta)$ such that the induced map $Y\dashrightarrow V$ is a morphism, and let $(Y,\Gamma)$ be an lc pair such that $\Gamma$ is a $\mathbb{Q}$-divisor and we can write $K_{Y}+\Gamma=f^{*}(K_{X}+\Delta)+E$ for an effective $f$-exceptional divisor $E$. 

Then, we have $\kappa_{\sigma}(G,K_{G}+\Gamma_{G})=0$, where $G$ is a sufficiently general fiber of the morphism $Y\to V$ and $K_{G}+\Gamma_{G}=(K_{Y}+\Gamma)|_{G}$. 
\end{lem}

\begin{proof}

It is clear that $\kappa_{\sigma}(G,K_{G}+\Gamma_{G})\geq0$. 
We prove $\kappa_{\sigma}(G,K_{G}+\Gamma_{G})\leq0$. 
Note that $\kappa_{\sigma}(Y,K_{Y}+\Gamma)={\rm dim}Z$ by hypothesis. 
We denote the morphism $Y\to Z$ by $\psi$. 
Since the map $X\dashrightarrow V$ is the Iitaka fibration associated to $K_{X}+\Delta$, by construction of $(Y,\Gamma)$ there is an ample $\mathbb{Q}$-Cartier divisor $A\geq0$ on $V$ and an effective $\mathbb{Q}$-divisor $N$ on $Y$ such that $f^{*}(K_{Y}+\Gamma)-\psi^{*}A\sim_{\mathbb{Q}}N$. 
Then 
$$K_{Y}+\Gamma\leq K_{Y}+\Gamma+k\psi^{*}A\sim_{\mathbb{Q}}(1+k)\psi^{*}A+N\leq (1+k)(\psi^{*}A+N)\sim_{\mathbb{Q}}(1+k)(K_{Y}+\Gamma)$$
 for any $k>0$. 
Therefore, we have 
$$\kappa_{\sigma}(Y,K_{Y}+\Gamma+k\psi^{*}A)=\kappa_{\sigma}(Y,K_{Y}+\Gamma)={\rm dim}V$$ for any $k>0$. 
Let $g\colon \bar{V}\to V$ be a resolution of $V$. 
By replacing $(Y,\Gamma)$ with a higher model, we may assume that the induced map $\bar{\psi}\colon Y\dashrightarrow \bar{V}$ is a morphism. 
Then the sufficiently general fiber of $\bar{\psi}$ is $G$. 
Fix an integer $k>0$ such that $K_{\bar{V}}+kg^{*}A$ is big and $kg^{*}A$ is a base point free Cartier divisor, and take $\bar{A}\sim kg^{*}A$ such that $(\bar{V}, \bar{A})$ and $(Y,\Gamma+\bar{\psi}^{*}\bar{A})$ are log smooth lc pairs. 
Then, by \cite[Theorem 2.1]{fujino-subadd-correct} (or (3.3) in \cite{fujino-subadd-correct}),
\begin{equation*}
\begin{split}
\kappa_{\sigma}(Y,K_{Y}+\Gamma+\bar{\psi}^{*}\bar{A})&\geq \kappa_{\sigma}(G,(K_{Y}+\Gamma+\bar{\psi}^{*}\bar{A})|_{G})+\kappa(\bar{V},K_{\bar{V}}+\bar{A} )\\
&=\kappa_{\sigma}(G,K_{G}+\Gamma_{G})+{\rm dim}\bar{V}
\end{split}
\end{equation*}
Since we have $\kappa_{\sigma}(Y,K_{Y}+\Gamma+k\psi^{*}A)=\kappa_{\sigma}(Y,K_{Y}+\Gamma)={\rm dim}V$ and $\bar{\psi}^{*}\bar{A}\sim k\psi^{*}A$, we have $\kappa_{\sigma}(G,K_{G}+\Gamma_{G})\leq0$. 
Therefore, we obtain $\kappa_{\sigma}(G,K_{G}+\Gamma_{G})=0$. 
\end{proof}

\begin{rem}
Quite recently, Lesieutre \cite{lesieutre} gave an example of an $\mathbb{R}$-divisor $D$ on a smooth projective variety $X$ such that $\kappa_{\sigma}(X,D)\neq \kappa_{\nu}(X,D)$, where $\kappa_{\nu}(X,\,\cdot\,)$ is an other notion of numerical dimension defined by numerical domination. 
So Lemma \ref{lem--iitakafib} does not follow from \cite[Theorem 6.1 (1)]{leh} or \cite[V, 4.2 Corollary]{nakayama}. 
Fortunately, argument by Nakayama \cite[Section V]{nakayama} works in the proof of Lemma \ref{lem--iitakafib}. 
For details, see \cite{fujino-subadd-correct}. 

Note that Lemma \ref{lem--iitakafib} holds true even if $\Delta$ is an $\mathbb{R}$-divisor. 
In fact, the $\mathbb{R}$-boundary divisor case of Lemma \ref{lem--iitakafib} follows from the argument of Shokurov polytopes and proof of Lemma \ref{lem--iitakafib}. 
\end{rem}

We introduce the notion of log smooth model. 

\begin{defn}[Log smooth model]\label{defnlogsmoothmodel}
Let $(X,\Delta)$ be an lc pair, and let  $f\colon W \to X$ be a log resolution of $(X,\Delta)$. 
Let $\Gamma$ be a boundary $\mathbb{R}$-divisor on $W$ such that $(W,\Gamma)$ is log smooth. 
Then $(W,\Gamma)$ is a {\it log smooth model} of $(X,\Delta)$ if we write 
$$K_{W}+\Gamma=f^{*}(K_{X}+\Delta)+E, $$
then
\begin{enumerate}
\item[$\bullet$]
$E$ is an effective $f$-exceptional divisor, and 
\item[$\bullet$] 
any $f$-exceptional prime divisor $E_{i}$ satisfying $a(E_{i},X,\Delta)>-1$ is a component of $E$ and $\Gamma-\llcorner \Gamma \lrcorner$.  
\end{enumerate}
\end{defn}

\begin{rem}\label{remlogsmooth}
When $\Delta$ is a $\mathbb{Q}$-divisor and $f\colon W \to X$ is a log resolution of $(X,\Delta)$, we can find a $\mathbb{Q}$-divisor $\Gamma$ on $W$ such that $(W,\Gamma)$ is a log smooth model of $(X,\Delta)$. 

Let $(X,\Delta)$ be an lc pair and $f\colon(W,\Gamma)\to (X,\Delta)$ be a log smooth model.  
Then we have $\kappa(X,K_{X}+\Delta)=\kappa(W,K_{W}+\Gamma)$ and $\kappa_{\sigma}(X,K_{X}+\Delta)=\kappa_{\sigma}(W,K_{W}+\Gamma)$
\end{rem}

In the proof of Theorem \ref{thmmain}, we use a special kind of log smooth model. 

\begin{lem}[{\cite[Lemma 2.10]{has-trivial}}]\label{lemlogresol}
Let $\pi \colon X \to Z$ be a projective morphism from a normal variety to a variety. 
Let $(X,\Delta)$ be an lc pair. 
Then there is a log smooth model $f\colon(W,\Gamma)\to(X,\Delta)$ such that 
\begin{enumerate}
\item[(i)]
$\Gamma=\Gamma'+\Gamma''$, where $\Gamma'\geq0$ and $\Gamma''$ is a reduced divisor or $\Gamma''=0$,
\item[(ii)]
$(\pi \circ f)({\rm Supp}\,\Gamma'') \subsetneq Z$, and 
\item[(iii)]
every lc center of $(W,\Gamma-t\Gamma'')$ dominates $Z$ for any $t\in (0,1]$.
\end{enumerate} 
\end{lem}

We introduce two results on the canonical bundle formula and the relative log MMP, which are proved by Fujino--Gongyo and Hacon--Xu, respectively. 

\begin{thm}[{\cite[Corollary 3.2]{fg-bundle}}]\label{thmcbf}
Let $\pi\colon X\to Z$ be a contraction of normal projective varieties, and let $(X,\Delta)$ be an lc pair such that $\Delta$ is a $\mathbb{Q}$-divisor. 
Suppose that $K_{X}+\Delta\sim_{\mathbb{Q}}\pi^{*}D$ for some $\mathbb{Q}$-divisor $D$ on $Z$ and all lc centers of $(X,\Delta)$ dominate $Z$. 

Then, there is a $\mathbb{Q}$-divisor $\Psi$ on $Z$ such that 
\begin{itemize}
\item
$(Z,\Psi)$ is klt, and
\item
$D\sim_{\mathbb{Q}}K_{Z}+\Psi$. 
\end{itemize}
\end{thm}

\begin{thm}[cf.~{\cite[Theorem 1.1]{haconxu-lcc}}]\label{thmmmp}
Let $\pi\colon X\to Z$ be a projective contraction of normal quasi-projective varieties, and let $(X,\Delta)$ be an lc pair. 
Let $F$ be a general fiber of $\pi$, and we set $K_{F}+\Delta_{F}=(K_{X}+\Delta)|_{F}$. 
Suppose that $(F,\Delta_{F})$ has a good minimal model and all lc centers of $(X,\Delta)$ dominate $Z$. 

Then, there is a good minimal model of $(X,\Delta)$ over $Z$.  
\end{thm}

We close this section with the following theorem proved by Kov\'acs and Patakfalvi, which plays a crucial role in the proof of Theorem \ref{thmmain}.

\begin{thm}[cf.~{\cite[Theorem 9.9]{kp}}]\label{thmkp}
Let $f\colon X\to Y$ be a contraction from a normal projective variety to a smooth projective variety, and let $(X,\Delta)$ be an lc pair such that $\Delta$ is a $\mathbb{Q}$-divisor. 
Suppose that the divisor $K_{F}+\Delta_{F}=(K_{X}+\Delta)|_{F}$ is big, where $F$ is a sufficiently general fiber of $f$. 
Let $M$ be a $\mathbb{Q}$-Cartier divisor on $Y$ such that $\kappa(Y,M)\geq 0$. 
Then we have
$$\kappa(X,K_{X}+\Delta-f^{*}K_{Y}+f^{*}M)\geq \kappa(F,K_{F}+\Delta_{F})+\kappa(Y,M).$$
\end{thm}

\begin{proof}
By taking a log smooth model of $(X,\Delta)$, we may assume that $(X,\Delta)$ is log smooth. 
Then the theorem is nothing but \cite[Theorem 9.9]{kp}. 
\end{proof}

\section{Proof of Theorem \ref{thmmain}, Theorem \ref{thmlogiitaka} and Corollary \ref{coriitaka}}\label{sec3}

In this section, we prove Theorem \ref{thmmain}, Theorem \ref{thmlogiitaka} and Corollary \ref{coriitaka}. 

The following lemma is important for the proof of Theorem \ref{thmmain}. 

\begin{lem}\label{lemkltbig}
To prove Theorem \ref{thmmain}, we can assume that 
$(X,\Delta)$ is klt and 
$K_{F}+\Delta_{F}$ is big. 
\end{lem}

\begin{proof}
The argument is very similar to \cite[Proof of Theorem 3.6]{huiitaka}. 
Let $f \colon X\to Y$, $\Delta$, $F$, and $K_{F}+\Delta_{F}$ be as in Theorem \ref{thmmain}. 
We prove Lemma \ref{lemkltbig} with several steps. 

\begin{step}\label{step1}
By taking a log smooth model of $(X,\Delta)$, we may assume that $(X,\Delta)$ is log smooth (see Remark \ref{remlogsmooth}). 
We take the relative Iitaka fibration of $K_{X}+\Delta$ over $Y$, which we denote $\phi \colon X\dashrightarrow Z$. 
By replacing $(X,\Delta)$ with a higher model, we can assume that $\phi$ is a morphism. 
Moreover, replacing $(X,\Delta)$ by a log smooth model as in Lemma \ref{lemlogresol}, we can assume that 
\begin{itemize}
\item[(i)]
$\Delta=\Delta_{1}+\Delta_{2}$, where $\Delta_{1}\geq0$ and $\Delta_{2}$ is a reduced divisor or $\Delta_{2}=0$,
\item[(ii)]
$\phi({\rm Supp}\,\Delta_{2}) \subsetneq Z$, and 
\item[(iii)]
any lc center of $(X,\Delta-t\Delta_{2})$ dominates $Z$ for any $t\in (0,1]$.
\end{itemize}
Note that the restriction of $\phi$ to $F$ is the Iitaka fibration of $K_{F}+\Delta_{F}$.  
Let $G$ be a sufficiently general fiber of $\phi$, and we set $K_{G}+\Delta_{G}=(K_{X}+\Delta)|_{G}$.
Since $K_{F}+\Delta_{F}$ is abundant, by Lemma \ref{lem--iitakafib}, we have $\kappa_{\sigma}(G,K_{G}+\Delta_{G})=0$. 
\end{step}

\begin{step}\label{step2}
In this step, we show that we can assume all lc centers of $(X,\Delta)$ dominate $Z$. 

By construction of the relative Iitaka fibration, there is a $\mathbb{Q}$-divisor $A$ on $Z$, which is ample over $Y$, such that $K_{F}+\Delta_{F}-(\phi^{*}A)|_{F}$ is $\mathbb{Q}$-linearly equivalent to an effective divisor and 
$\kappa(F,(\phi^{*}A)|_{F})=\kappa(F, K_{F}+\Delta_{F})$. 
Since $\phi({\rm Supp}\,\Delta_{2}) \subsetneq Z$ by (ii) in Step \ref{step1}, we can find a sufficiently small positive rational number $t$ such that $(\tfrac{1}{2}\phi^{*}A-t\Delta_{2})|_{F}$ is $\mathbb{Q}$-linearly equivalent to an effective divisor.  
Then, we have
\begin{equation*}
\begin{split}
\kappa(F,K_{F}+\Delta_{F})&\geq\kappa(F,(\phi^{*}A)|_{F})\geq 
\kappa\bigl(F, (\phi^{*}A-t\Delta_{2})|_{F}\bigr)\\
&=\kappa\bigl(F, \tfrac{1}{2}(\phi^{*}A)|_{F}+(\tfrac{1}{2}\phi^{*}A-t\Delta_{2})|_{F}\bigr)\\
&\geq\kappa\bigl(F, \tfrac{1}{2}(\phi^{*}A)|_{F}\bigr)\\
&=\kappa(F, K_{F}+\Delta_{F}). 
\end{split}
\end{equation*}
So $\kappa(F, K_{F}+\Delta_{F}-t\Delta_{2}|_{F})=\kappa(F, K_{F}+\Delta_{F})$. 

First, suppose that Theorem \ref{thmmain} (1) holds for the morphism $(X,\Delta-t\Delta_{2})\to Y$. 
Then 
\begin{equation*}
\begin{split}
\kappa(X,K_{X}+\Delta-f^{*}K_{Y}+f^{*}M)&\geq\kappa(X,K_{X}+\Delta-t\Delta_{2}-f^{*}K_{Y}+f^{*}M)\\&\geq \kappa(F, K_{F}+\Delta_{F}-t\Delta_{2}|_{F})+\kappa(Y,M)\\
&= \kappa(F, K_{F}+\Delta_{F})+\kappa(Y,M). 
\end{split}
\end{equation*}
Therefore, we see that Theorem \ref{thmmain} (1) holds for the morphism $f \colon (X,\Delta)\to Y$. 

Next, suppose that Theorem \ref{thmmain} (1) holds for the morphism $(X,\Delta-t\Delta_{2})\to Y$ and Theorem \ref{thmmain} (2) holds for the morphism $(X,\Delta-\tfrac{1}{2}t\Delta_{2})\to Y$. 
By the above calculation, we obtain 
$$\kappa(X,K_{X}+\Delta-t\Delta_{2}-f^{*}K_{Y}+f^{*}M)\geq \kappa(F, K_{F}+\Delta_{F})+\kappa(Y,M)\geq0.$$
Then, by Lemma \ref{lemdim} and Theorem \ref{thmmain} (2) for $(X,\Delta-\tfrac{1}{2}t\Delta_{2})\to Y$, we have 
\begin{equation*}
\begin{split}
&\kappa(X,K_{X}+\Delta-f^{*}K_{Y}+f^{*}M)\\
=&\kappa(X,K_{X}+\Delta-\tfrac{1}{2}t\Delta_{2}-f^{*}K_{Y}+f^{*}M)
=\kappa_{\sigma}(X,K_{X}+\Delta-\tfrac{1}{2}t\Delta_{2}-f^{*}K_{Y}+f^{*}M)\\
=&\kappa_{\sigma}(X,K_{X}+\Delta-f^{*}K_{Y}+f^{*}M).
\end{split}
\end{equation*}
Therefore, $K_{X}+\Delta-f^{*}K_{Y}+f^{*}M$ is abundant. 

In this way, to prove Theorem \ref{thmmain} for the morphism $f \colon (X,\Delta)\to Y$, it is sufficient to prove Theorem \ref{thmmain} (1) for the morphism $(X,\Delta-t\Delta_{2})\to Y$ and Theorem \ref{thmmain} (2) for the morphism $(X,\Delta-\tfrac{1}{2}t\Delta_{2})\to Y$. 
By (iii) in Step \ref{step1}, all lc centers of $(X,\Delta-t\Delta_{2})$ and those of $(X,\Delta-\tfrac{1}{2}t\Delta_{2})$ dominate $Z$. 
Furthermore, by (ii) in Step \ref{step1}, we have $\Delta_{2}|_{G}=0$. 
Since  $\kappa_{\sigma}(G,K_{G}+\Delta_{G})=0$, we have 
$$\kappa_{\sigma}(G,K_{G}+\Delta_{G}-t\Delta_{2}|_{G})=\kappa_{\sigma}(G,K_{G}+\Delta_{G}-\tfrac{1}{2}t\Delta_{2}|_{G})=0.$$ 
Therefore, in any case, by replacing $\Delta$ with $\Delta-t\Delta_{2}$ or $\Delta-\tfrac{1}{2}t\Delta_{2}$, we may assume that any lc center of $(X,\Delta)$ dominates $Z$. 
\end{step}

\begin{step}\label{step3}
From this step, we do not use the fact that $(X,\Delta)$ is log smooth. 
Now we have the following diagram
$$ \xymatrix@R=12pt 
{ (X,\Delta) \ar[dd]_{f} \ar[dr]^{\phi}\\
&Z\ar[dl] \\ 
Y }$$
such that 
\begin{itemize}
\item
$\kappa(F,K_{F}+\Delta_{F})={\rm dim}Z-{\rm dim}Y$, 
\item
$\kappa_{\sigma}(G,K_{G}+\Delta_{G})=0$, where $G$ is a sufficiently general fiber of $\phi$, and 
\item
any lc center of $(X,\Delta)$ dominates $Z$. 
\end{itemize}
In this step, we show that we can assume that $K_{X}+\Delta$ is semi-ample over $Z$. 

By the second condition and results in \cite{gongyo1}, the pair $(G,\Delta_{G})$ has a good minimal model. 
By the third condition and Theorem \ref{thmmmp}, there is a good minimal model $(X',\Delta')$ of $(X,\Delta)$ over $Z$. 
We denote the natural morphism $X'\to Y$ by $f'$, and let $F'$ be a sufficiently general fiber of $f'$. 
We define $K_{F'}+\Delta_{F'}$ by $K_{F'}+\Delta_{F'}=(K_{X'}+\Delta')|_{F'}$. 
By the negativity lemma, we obtain
\begin{equation*}
\begin{split}
\kappa(X,K_{X}+\Delta-f^{*}K_{Y}+f^{*}M)&=\kappa(X', K_{X'}+\Delta'-f'^{*}K_{Y}+f'^{*}M)\quad {\rm and}\\
\kappa(F,K_{F}+\Delta_{F})&=\kappa(F', K_{F'}+\Delta_{F'}),
\end{split}
\end{equation*} 
and that $K_{X}+\Delta-f^{*}K_{Y}+f^{*}M$ is abundant if and only if $K_{X'}+\Delta'-f'^{*}K_{Y}+f'^{*}M$ is abundant. 
Furthermore, by construction, every lc center of $(X',\Delta')$ dominates $Z$ and the restriction of $K_{X'}+\Delta'$ to a sufficiently general fiber of the morphism $X'\to Z$ is $\mathbb{Q}$-linearly trivial. 
From these facts, we may replace $(X,\Delta)$ by $(X',\Delta')$, and  we can assume that $K_{X}+\Delta$ is semi-ample over $Z$. 
\end{step}

\begin{step}\label{step4}
With this step we complete the proof. 

Let $\psi\colon X\to X''$ be the contraction over $Z$ induced by $K_{X}+\Delta$. 
There is a $\mathbb{Q}$-divisor $D$ on $X''$ such that $K_{X}+\Delta \sim_{\mathbb{Q}}\psi^{*}D$. 
From the assumption that $\kappa_{\sigma}(G,K_{G}+\Delta_{G})=0$, we have $K_{G}+\Delta_{G}\sim_{\mathbb{Q}}0$, hence we see that $\psi(G)$ is a point. 
Since $G$ is a sufficiently general fiber of $\phi$, the natural morphism $X''\to Z$ is birational. 
Since all lc centers of $(X,\Delta)$ dominate $Z$, all lc centers of $(X,\Delta)$ dominate $X''$. 
By Theorem \ref{thmcbf}, we can find a $\mathbb{Q}$-divisor $\Delta''$ on $X''$ such that $(X'',\Delta'')$ is klt and $D\sim_{\mathbb{Q}}K_{X''}+\Delta''$. 
If we denote the natural morphism $X''\to Y$ by $f''$, then we have
$$\kappa(X,K_{X}+\Delta-f^{*}K_{Y}+f^{*}M)=\kappa(X'',K_{X''}+\Delta''-f''^{*}K_{Y}+f''^{*}M)$$ 
and $K_{X}+\Delta-f^{*}K_{Y}+f^{*}M$ is abundant if and only if $K_{X''}+\Delta''-f''^{*}K_{Y}+f''^{*}M$ is abundant. 
Let $F''$ be a sufficiently general fiber of $f''$. 
Then $\psi\colon X\to X''$ induces a morphism $\psi_{F}\colon F\to F''$, and $\psi_{F}^{*}((K_{X''}+\Delta'')|_{F''})\sim_{\mathbb{Q}}K_{F}+\Delta_{F}$. 
So we have 
\begin{equation*}
\begin{split}
\kappa(F'',(K_{X''}+\Delta'')|_{F''})&=\kappa(F,K_{F}+\Delta_{F})={\rm dim}Z-{\rm dim}Y\\
&={\rm dim}X''-{\rm dim}Y={\rm dim}F''.
\end{split}
\end{equation*} 
From these facts, by replacing $(X,\Delta)$ with $(X'',\Delta'')$, we can assume that $(X,\Delta)$ is klt and $\kappa(F,K_{F}+\Delta_{F})={\rm dim}X-{\rm dim}Y={\rm dim}F$. 
\end{step}
So we complete the proof. 
\end{proof}

From now on, we prove Theorem \ref{thmmain}, Theorem \ref{thmlogiitaka} and Corollary \ref{coriitaka}. 

\begin{proof}[Proof of Theorem \ref{thmmain}]
Let $f \colon X\to Y$, $\Delta$, $F$, $K_{F}+\Delta_{F}$ and $M$ be as in Theorem \ref{thmmain}. 
By Lemma \ref{lemkltbig}, we may assume that $(X,\Delta)$ is klt and $K_{F}+\Delta_{F}$ is big. 
In this situation, Theorem \ref{thmmain} (1) immediately follows from Theorem \ref{thmkp}. 
Therefore, we only need to show Theorem \ref{thmmain} (2). 

Assume the existence of a good minimal model or a Mori fiber space for all projective  klt pairs of dimension $\leq n$, where $n={\rm dim}Y-\kappa(Y,M)$. 
We put $D=M-K_{Y}$, and suppose that $D$ is nef and abundant. 
By \cite[Proposition 6-1-3]{kmm-mmp}, there is a projective birational morphism $g\colon Y'\to Y$, a contraction $h\colon Y'\to Z$ and a nef and big $\mathbb{Q}$-divisor $D_{Z}$ on $Z$ such that $g^{*}D \sim_{\mathbb{Q}}h^{*}D_{Z}$. 
By replacing $(X,\Delta)$ with a log smooth model, we may assume that the induced map $f'\colon X\dashrightarrow Y'$ is a morphism. 
Since $D_{Z}$ is nef and big and since $(X,\Delta)$ is klt, we can find $\Delta'\sim_{\mathbb{Q}}f'^{*}h^{*}D_{Z}$ such that $(X,\Delta+\Delta')$ is klt. 
By construction, we have 
\begin{equation*}
\begin{split}
K_{X}+\Delta-f^{*}K_{Y}+f^{*}M=K_{X}+\Delta+f'^{*}g^{*}D=K_{X}+\Delta+f'^{*}h^{*}D_{Z}\sim_{\mathbb{Q}}K_{X}+\Delta+\Delta'.
\end{split}
\end{equation*}
In this way, it is sufficient to prove that $K_{X}+\Delta+\Delta'$ is abundant. 

By Theorem \ref{thmmain} (1), we have $\kappa(X,K_{X}+\Delta+\Delta')\geq0$. 
So we can construct the Iitaka fibration $X\dashrightarrow V$ of $K_{X}+\Delta+\Delta'$. 
By replacing $(X,\Delta+\Delta')$ with a log smooth model, we may assume that the map $X\dashrightarrow V$ is a morphism.  
Let $G$ be a sufficiently general fiber of the Iitaka fibration, and we set $K_{G}+\Delta_{G}+\Delta'_{G}=(K_{X}+\Delta+\Delta')|_{G}$. 
Then
\begin{equation*}
\begin{split}
{\rm dim}G={\rm dim}X-\kappa(X,K_{X}+\Delta+\Delta')\leq {\rm dim}Y-\kappa(Y,M),
\end{split}
\end{equation*}
where the second inequality follows from Theorem \ref{thmmain} (1) and bigness of $K_{F}+\Delta_{F}$. 
Since $(G,\Delta_{G}+\Delta'_{G})$ is klt, it has a good minimal model. 
By \cite[Theorem 4.3]{gongyolehmann}, we have $\kappa_{\sigma}(G, K_{G}+\Delta_{G}+\Delta'_{G})=\kappa(G, K_{G}+\Delta_{G}+\Delta'_{G})=0.$
By \cite[V, 2.7 (9) Proposition]{nakayama}, we obtain 
$$\kappa_{\sigma}(X, K_{X}+\Delta+\Delta')\leq \kappa_{\sigma}(G, K_{G}+\Delta_{G}+\Delta'_{G})+{\rm dim}V=\kappa(X,K_{X}+\Delta+\Delta').$$
Therefore, we see that $K_{X}+\Delta+\Delta'$ is abundant. 
So we are done. 
\end{proof}

\begin{proof}[Proof of Theorem \ref{thmlogiitaka}]
We may assume $\kappa(Y,K_{Y}+\Delta_{Y})\geq 0$ and $\kappa(F,K_{F}+\Delta_{F})\geq 0$. 
We prove the theorem using Theorem \ref{thmmain} (1). 
The strategy is the same as the proof of \cite[Proposition 9.12]{kp}. 
We freely use notations in Conjecture \ref{conjiitaka}. 

First, we apply \cite[Lemma 9.10]{kp} to the morphism $(X,\Delta_{X})\to Y$. 
There is a log resolution $\varphi\colon X'\to X$, which is an isomorphim over $f^{-1}(U)$ for some open subset $U\subset Y$, such that if we put $\Delta_{X'}$ as the support of $\varphi^{*}\Delta_{X}$ and if we put $\Delta_{X'}^{h}$ is the horizontal part of $\Delta_{X'}$ with respect to the morphism $f\circ \varphi\colon X'\to Y$, then 
\begin{itemize}
\item[($*$)]
for any prime divisor $P$ on $Y$, the pair $(X',{\rm Supp}(\Delta_{X'}^{h}+\varphi^{*}f^{*}P))$ is log smooth over a neighborhood of the generic point of $P$. 
\end{itemize} 
Since $\varphi_{*}(K_{X'}+\Delta_{X'})=K_{X}+\Delta_{X}$, we have $\kappa(X',K_{X'}+\Delta_{X'})\leq \kappa(X,K_{X}+\Delta_{X})$. 
By construction of $\varphi$, any sufficiently general fiber $F'$ of $f\circ \varphi$ is the same as $F$. 
Therefore, the divisor $K_{F'}+\Delta_{X'}|_{F'}$ is abundant. 
Furthermore, $\Delta_{X'}$ is reduced and simple normal crossing, and we have ${\rm Supp}\varphi^{*}f^{*}\Delta_{Y}\subset {\rm Supp}\Delta_{X'}$. 
In this way, we may replace $(X,\Delta_{X})$ by $(X',\Delta_{X'})$, and therefore we may assume condition ($*$) for $(X,\Delta_{X})$. 

Next, we apply \cite[Lemma 9.11]{kp} to the morphism $(X,\Delta_{X})\to Y$. 
We obtain the following diagram 
$$ \xymatrix
{ X \ar[d]_{f} &\overline{X} \ar[l]_{\psi} \ar[d]^{\overline{f}}\\
Y&\overline{Y} \ar[l]^{\tau}
}
$$
such that 
\begin{itemize}
\item
$\overline{X}$ and $\overline{Y}$ are smooth projective varieties, $\psi$ is a generically finite surjective morphism, $\tau$ is a finite surjective morphism, and $\overline{f}\colon \overline{X}\to \overline{Y}$ is a contraction, 
\item 
$\overline{f}$ agrees with the pullback of $f$ via an \'etale morphism over an open subset of $Y$, and 
\item
there is a reduced divisor $\Delta_{\overline{X}}$, which is a simple normal crossing divisor and agrees with the pullback of $\Delta_{X}$ over an open subset of $Y$, such that we have 
$$\kappa(\overline{X}, K_{\overline{X}}+\Delta_{\overline{X}}-\overline{f}^{*}K_{\overline{Y}}+\overline{f}^{*}\tau^{*}(K_{Y}+\Delta_{Y}))\leq \kappa(X,K_{X}+\Delta_{X}).$$
\end{itemize}
Let $\overline{F}$ be a sufficiently general fiber of $\overline{f}$, and put $\Delta_{\overline{F}}=\Delta_{\overline{X}}|_{\overline{F}}$. 
By the second condition and definition of $\Delta_{\overline{F}}$, we have 
$$\kappa(\overline{F}, K_{\overline{F}}+\Delta_{\overline{F}})=\kappa(F,K_{F}+\Delta_{F})\quad {\rm and} \quad \kappa_{\sigma}(\overline{F}, K_{\overline{F}}+\Delta_{\overline{F}})=\kappa_{\sigma}(F,K_{F}+\Delta_{F}).$$
So $K_{\overline{F}}+\Delta_{\overline{F}}$ is abundant. 
Since $\kappa(Y,K_{Y}+\Delta_{Y})\geq 0$ and $\kappa(F,K_{F}+\Delta_{F})\geq 0$, we can apply Theorem \ref{thmmain} (1) to the morphism $\overline{f}\colon (\overline{X},\overline{\Delta})\to \overline{Y}$ and $\tau^{*}(K_{Y}+\Delta_{Y})$. 
With the third condition, we obtain 
\begin{equation*}
\begin{split}
\kappa(X,K_{X}+\Delta_{X})&\geq \kappa(\overline{X}, K_{\overline{X}}+\Delta_{\overline{X}}-\overline{f}^{*}K_{\overline{Y}}+\overline{f}^{*}\tau^{*}(K_{Y}+\Delta_{Y}))\\
&\geq\kappa(\overline{F}, K_{\overline{F}}+\Delta_{\overline{F}})+\kappa(\overline{Y}, \tau^{*}(K_{Y}+\Delta_{Y}))\\
&=\kappa(F,K_{F}+\Delta_{F})+\kappa(Y, K_{Y}+\Delta_{Y}). 
\end{split}
\end{equation*}
In this way, we see that Conjecture \ref{conjiitaka} holds for the morphism $X\to Y$ and the lc pairs $(X,\Delta_{X})$ and $(Y,\Delta_{Y})$. 
\end{proof}

\begin{proof}[Proof of Corollary \ref{coriitaka}]
We freely use notations as in Conjecture \ref{conjiitaka}. 
We may assume $\kappa(F,K_{F}+\Delta_{F})\geq0$ and $\kappa(Y,K_{Y}+\Delta_{Y})\geq0$. 
When ${\rm dim}X-{\rm dim}Y\leq 3$, Conjecture \ref{conjiitaka} follows from Theorem \ref{thmlogiitaka}. 
We assume ${\rm dim}X\leq5$. 
If ${\rm dim}Y\geq 2$, then ${\rm dim}X-{\rm dim}Y\leq 3$ and so Conjecture \ref{conjiitaka} holds in this case.
From now on, we assume ${\rm dim}Y=1$. 
We can assume $\kappa(F,K_{F}+\Delta_{F})=0$ because otherwise $(F,\Delta_{F})$ has a good minimal model (see \cite[Lemma 3.8]{birkar1} and \cite[Theorem 3.3]{fujino-lcring}). 
We can also assume $\kappa(Y,K_{Y}+\Delta_{Y})=0$ because otherwise Corollary \ref{coriitaka} follows from \cite[Theorem 1.9]{fujino-noteiitaka}. 
Then $Y$ is an elliptic curve and $\Delta_{Y}=0$, or $Y=\mathbb{P}^{1}$ and $\Delta_{Y}\sim -K_{Y}$. 
It is sufficient to prove the inequality $\kappa(X,K_{X}+\Delta_{X})\geq0$. 
When $Y$ is an elliptic curve and $\Delta_{Y}=0$, the inequality 
 follows from \cite[Lemma 2.9]{wang}. 
So assume $Y=\mathbb{P}^{1}$ and $\Delta_{Y}\sim -K_{Y}$. 

By taking a suitable coordinate of $\mathbb{P}^{1}$, we may assume $\Delta_{Y}=P_{0}+P_{\infty}$, where $P_{0}$ and $P_{\infty}$ are prime divisors corresponding to $0$ and $\infty$, respectively. 
Let $n$ be a common multiple of all coefficients of $f^{*}\Delta_{Y}$, and we define $g\colon Y\to Y$ by $g(x)=x^{n}$. 
It is clear that $g$ is \'etale over $Y\setminus\{0,\infty \}$ and $K_{Y}+\Delta_{Y}=g^{*}(K_{Y}+\Delta_{Y})$. 
Let $X'$ be the normalization of the main component of $X\times_{Y}Y$, and we denote natural morphisms $X'\to Y$ and $X'\to X$ by $f'$ and $h$. 
We define $\Delta'$ by $K_{X'}+\Delta'=h^{*}(K_{X}+\Delta_{X})$. 
Since ${\rm Supp}f^{*}\Delta_{Y}\subset {\rm Supp}\Delta_{X}$ and $g$ ramifies only over ${\rm Supp}\Delta_{Y}$, we see that $h$ does not ramify over $X\setminus{\rm Supp}\Delta_{X}$. 
So $\Delta'\geq 0$, thus $\Delta'$ is reduced and $(X',\Delta')$ is lc. 
Moreover, argument in \cite[Proposition 7.23]{kollar-mori} shows that $f'^{*}\Delta_{Y}$ is reduced, which shows $f'^{*}\Delta_{Y}\leq \Delta'$. 
By construction, it is easy to check that we only have to prove $\kappa(X',K_{X'}+\Delta')\geq0$. 
In this way, replacing $(X,\Delta_{X})$ and $f$ by $(X',\Delta')$ and $f'$ respectively, we may assume  $f^{*}\Delta_{Y}\leq \Delta_{X}$. 
We note that after replacing $(X,\Delta_{X})$ and $f$ the pair $(X,\Delta_{X})$ is not necessarily log smooth, but it is harmless because we only use that $(X,\Delta_{X})$ is lc. 

We set $D=\Delta_{X}-f^{*}\Delta_{Y}\geq 0$. 
Then $(X,D)$ is lc and $K_{X}+\Delta_{X}=K_{X}+D-f^{*}K_{Y}$. 
Since $\kappa(F,K_{F}+\Delta_{F})=0$ and $Y=\mathbb{P}^{1}$, we can write $f_{*}\mathcal{O}_{X}(m(K_{X}+\Delta_{X}))\simeq \mathcal{O}_{Y}(a)$ for some $m>0$ and $a\in \mathbb{Z}$. 
By weak positivity theorem \cite[Theorem 1.1]{fujino-noteiitaka}, we have $a\geq 0$, which implies $\kappa(X,K_{X}+\Delta_{X})\geq0$ because $Y=\mathbb{P}^{1}$. 
So we are done. 
\end{proof}


\end{document}